\documentclass{amsart}

\RequirePackage[abspath,realmainfile]{currfile}

\usepackage{amsmath,amssymb,amsthm,amsfonts,amsbsy}
\usepackage{graphicx}
 
\usepackage{enumerate}
\usepackage[final]
{showkeys}

\usepackage{hyperref}

\providecommand{\texorpdfstring}[2]{#1}
\providecommand{\url}[1]{#1}

\renewcommand{\le}{\leqslant}
\renewcommand{\ge}{\geqslant}

\renewcommand{\P}{\operatorname{\mathsf{P}}} 
\newcommand{\E}{\operatorname{\mathsf{E}}}

\newcommand{\ii}[1]{\operatorname{\mathsf{I}}\{#1\}}

\newcommand{\R}{\mathbb{R}}

\newcommand{\al}{\alpha}
\newcommand{\be}{\beta}
\newcommand{\ga}{\gamma}

\newcommand{\de}{\delta}

\renewcommand{\th}{\theta}

\newcommand{\la}{\lambda}

\newcommand{\Si}{\Sigma}

\newcommand{\tf}{\tilde %f
w}

\newcommand{\ttt}{\tilde t}
\newcommand{\tuu}{\tilde\uu}

\newcommand{\tx}{\tilde\xx}
\newcommand{\tX}{\tilde\X}

\newcommand{\vpi}{\varphi}

\newcommand{\e}{{\mathbf{e}}}
\newcommand{\uu}{{\mathbf{u}}}
\newcommand{\vv}{{\mathbf{v}}}

\newcommand{\x}{{\mathbf{x}}}
\newcommand{\y}{{\mathbf{y}}}
\newcommand{\z}{{\mathbf{z}}}
\newcommand{\ZZ}{{\mathbf{Z}}}
\newcommand{\xx}{\mathbf{x}}
\newcommand{\X}{\mathbf{X}}
\newcommand{\Y}{\mathbf{Y}}

\newcommand{\muu}{{\boldsymbol{\mu}}}

\newcommand{\0}{\mathbf{0}}

\newcommand{\ip}[2]{\langle #1,#2\rangle}

\newcommand{\w}{w}

\newtheorem*{theorem*}{Theorem~\ref{th:}'}
\newtheorem*{thm*}{Theorem~\ref{th:der}'}
\newtheorem*{thA}{Theorem~A}
\newtheorem*{thB}{Theorem~B}
\newtheorem{theorem}{Theorem}
\newtheorem{corollary}[theorem]{Corollary}
\newtheorem*{corollary*}{Corollary~\ref{cor:A}'}
\newtheorem*{cor*}{Corollary~\ref{cor:Ader}'}
\newtheorem{lemma}[theorem]{Lemma}
\newtheorem{proposition}[theorem]{Proposition}
\theoremstyle{remark}
\newtheorem{remark}[theorem]{Remark}

\numberwithin{equation}{section}
\numberwithin{theorem}{section}

\begin{document}

\title{Exact lower and upper bounds for shifts of Gaussian measures}

\author{Iosif Pinelis}
\address{Department of Mathematical Sciences\\Michigan Technological University\\Hough\-ton, Michigan 49931}
\email{ipinelis@mtu.edu}
%\urladdr{www.math.sc.edu/$\sim$howard} % Delete if not wanted.



\date{\today}

\begin{abstract}
Exact upper and lower bounds on the ratio  
$\E\w(\X-\vv)/\E\w(\X)$ for a centered Gaussian random vector $\X$ in $\R^n$,  
as well as bounds on the rate of change of  
$\E\w(\X-t\vv)$ in $t$, where $\w\colon\R^n\to[0,\infty)$ is any even unimodal function and $\vv$ is any vector in $\R^n$.  
As a corollary of such results, exact upper and lower bounds on the power function of statistical tests for the mean of a multivariate normal distribution are given. 
\end{abstract}

\maketitle

%%
%% LaTeX can automatically make a table of contents.  This is done by
%% uncommenting the following:
%%

%\tableofcontents

\section{Introduction}
\label{intro}

%Let $\w\colon\R^n\to[0,\infty)$ be an even unimodal function, so that the set \break 
%$\{\x\in\R^n\colon \w(\x)>c\}$ is convex for every real $c$, and $\w(-\x)=\w(\x)$ for all $\x\in\R^n$. 
%Also, let $A$ be a symmetric subset of $\R^n$, so that $-A=A$. 
%Note that 
%\begin{equation}\label{eq:mix}
%	\w(\x)=\int_0^\infty \w_c(\x)\,dc
%\end{equation}
%for all $\x\in\R^n$, where 
%\begin{equation}\label{eq:A_w,c}
%	\w_c:=\mathsf{I}_{A_{\w,c}},\quad A_{\w,c}:=\{\x\in\R^n\colon\w(\x)%\ge 
%	>c\}, 
%\end{equation}
%and $\mathsf{I}$ denotes the indicator. 
%Note also that the functions $\w_c$ are log concave. 
%Thus, the unimodal function $\w$ is a mixture of log concave functions $\w_c$; moreover, the functions $\w_c$ are even whenever the functions $\w$ is even (that is, whenever $\w(-\x)=\w(\x)$ for all $\x\in\R^n$). 

%\item $\w\colon\R^n\to[0,\infty)$ is a unimodal function, which means, by definition, that the set $\{\x\in\R^n\colon \w(\x)\ge c\}$ is convex for every real $c$; 
%
%Let $\w\colon\R^n\to[0,\infty)$ be an even unimodal function and let $A$ be a symmetric subset of $\R^n$, so that $\w(-\x)=\w(\x)$ for all $\x\in\R^n$ and $-A=A$. 
A classic %theorem 
result due to Anderson \cite[Theorem~1]
{anderson55} states the following: If \break 
$\w\colon\R^n\to[0,\infty)$ is an even unimodal function and if 
%a symmetric nonnegative function $f$ on $\R^n$ is such that the set $\{\x\in\R^n\colon f(\x)\ge c\}$ is convex for every real $c$ and if 
$A$ is a symmetric convex subset of $\R^n$,
%the symmetric set $A$ is convex, 
then %\break 
$\int_A \w(\xx+t\vv)\,d\x$ is nondecreasing in $t\ge0$, for any vector $\vv=(v_1,\dots,v_n)\in\R^n$. 

By a clever application of Anderson's theorem, Marshall and Olkin~\cite{marsh-olkin74} showed that, if a random vector $\X$ in $\R^n$ has a Schur-concave density and if the indicator of a subset $A$ of $\R^n$ is Schur concave and permutation symmetric, then \break 
$\P(\X\in\vv+A)$ is Schur concave in $\vv$.  

By using a rather different method, it was shown in \cite{schur2_published} that for a standard %normal 
Gaussian random vector $\ZZ$ in $\R^n$ the probability $\P(\ZZ\in \vv+A)$ is Schur concave/Schur convex in $(v_1^2,\dots,v_n^2)$ provided that the indicator of the set $A$ is so, respectively. An application of this result, also given in \cite{schur2_published}, was that, for large $n$, tests whose rejection regions are balls, centered at the origin, with respect to the $\ell_p$-norm on $\R^n$ with $p>2$ will be generally preferable to the likelihood ratio test in terms of the asymptotic relative efficiency.   

Here we obtain a number of results somewhat related to the just mentioned ones, including exact upper and lower bounds on the ratio  
$\E\w(\X-\vv)/\E\w(\X)$ for a centered Gaussian random vector $\X$ in $\R^n$, 
% and a symmetric convex subset $A$ of $\R^n$, 
as well as bounds on the rate of change of  
$\E\w(\X-t\vv)$ in $t$; here, again, $\w\colon\R^n\to[0,\infty)$ is any even unimodal function. 
%More generally, certain weighted versions of these results are obtained. 
As a corollary of such results, we give exact upper and lower bounds on the power function of statistical tests for the mean of a multivariate normal distribution. 
%It should be possible to use the bounds presented in this paper in conjunction with the mentioned results in \cite{anderson55}, \cite{marsh-olkin74}, \cite{schur2_published}. 

The proof of cited Theorem~1 in \cite%[Theorem~1]
{anderson55} was based on the Brunn--Minkowski inequality. However (as can be seen from Remark~\ref{rem:mix} in the present paper), 
%in view of \eqref{eq:mix}, 
this theorem %is 
can be immediately reduced to the case when the function $\w$ is log concave, and then \cite[Theorem~1]{anderson55} can be obtained at once from the following version of the Pr\'ekopa--Leindler theorem -- cf.\ \cite[Corollary~3.5]{brasc-lieb76}: 

\begin{thA}%\label{th:A}
If a function $F\colon\R^m\times\R^n\to[0,\infty]$ is log concave, then the function $G\colon\R^m\to[0,\infty]$ given by the formula 
\begin{equation*}
	G(\x):=\int_{\R^n}F(\x,\y)\,d\y
\end{equation*}
for $\x\in\R^n$ is also log concave. 
\end{thA}

In turn, as shown in \cite{brasc-lieb76}, the Pr\'ekopa--Leindler theorem follows already from the simple ``one-dimensional'' case of the Brunn--Minkowski inequality, for subsets of $\R$; concerning this ``one-dimensional'' case, see e.g.\ \cite[Theorem 2.1]{gardner02}. It is also shown in \cite[Corollary~3.4]{brasc-lieb76} that, vice versa, the Brunn--Minkowski inequality follows from a generalized version of the Pr\'ekopa--Leindler theorem.  

Theorem~A will be the main tool in the proof of the mentioned exact upper and lower bounds on the ratio  
$\E\w(\X-\vv)/\E\w(\X)$. 
Another ingredient, which significantly simplifies the proof, is a so-called special-case l'Hospital-type rule for monotonicity (cf.\ e.g.\ \cite[Proposition 4.1]{pin06}): 

\begin{thB}%\label{th:B}
Let $-\infty\le a<b\le\infty$. Let $f$ and $g$ be differentiable functions defined on the interval $(a, b)$ such that $g$ and $g'$ do not take on the zero value
and do not change their respective signs on $(a, b)$. Suppose also that $f(a+)=g(a+)=0$ or $f(b-)=g(b-)=0$. 
%Consider the ratio $r:=f/g$ and the ``derivative'' ratio $\rho:=f'/g'$. If $\rho$ is increasing, then $r$ is increasing as well. 
Under these conditions, if the ``derivative'' ratio $f'/g'$ is increasing on $(a, b)$, 
then the ratio $f/g$ is so as well.  
\end{thB} 

General versions of this l'Hospital-type rule for monotonicity, without the assumption that $f(a+)=g(a+)=0$ or $f(b-)=g(b-)=0$ are also known; see again \cite{pin06} and references therein. 

%ThA: Pr\'ekopa--Leindler 
%
%\cite{brasc-lieb76} Brunn--Minkowski
%
%\hrule
%
%ThB: l'Hospital
%
%Gromov, Anderson -- l'Hospital; greatly simplifies proofs
%
%\hrule
%
%Anderson '55

\medskip
\hrule
\medskip

Here are %general??? 
notations used in the rest of this paper: 
\begin{itemize}
%\item $\w\colon\R^n\to[0,\infty)$ is a unimodal function, which means, by definition, that the set $\{\x\in\R^n\colon \w(\x)\ge c\}$ is convex for every real $c$; 
	\item $\X$ is a zero-mean Gaussian random vector in $\R^n$ with a nonsingular covariance matrix $\Si$; 
	\item $\ZZ$ is a zero-mean Gaussian random vector in $\R^n$ with covariance matrix $I_n$;
	\item $\ga_n$ is the standard Gaussian measure over $\R^n$; 
	\item unless otherwise stated, $\w\colon\R^n\to[0,\infty)$ is any even %log-concave 
unimodal %(weight) 
function %, which is not identically $0$
such that $\E\w(\ZZ)>0$ (and hence $\E\w(\X)>0$, in view of the absolute continuity of the distribution of $\ZZ$ with respect to that of $\X$); recall here that the unimodality of the function $\w$ means that the set $\{\x\in\R^n\colon \w(\x)>c\}$ is convex for each real $c$; 
	\item unless otherwise stated, $A$ is any %non-empty 
	symmetric convex subset of $\R^n$ 
	such that $\P(\ZZ\in A)>0$ (and hence $\P(\X\in A)>0$); 
	\item $\ip\cdot\cdot$ denotes the standard inner product over $\R^n$, and $\|\cdot\|$ denotes the Euclidean norm on $\R^n$;
	\item $\de^*(\cdot|A)$ is the support function of a set $A\subseteq\R^n$, given by the formula 
\begin{equation}\label{eq:de*}
	\de^*(\vv|A):=\sup\{\ip\z\vv\colon\z\in A\}
\end{equation}
for $\vv\in\R^n$ (cf.\ e.g.\ \cite[page~28]{rocka}); 
\item $\Phi$ is the standard normal cumulative distribution function, and $\vpi=\Phi'$ is the standard normal density function; 
 \item for $t\in[0,\infty)$ and $a\in[0,\infty]$,  
\begin{equation}\label{eq:r}
	r_t(a):=
	\left\{
	\begin{alignedat}{2}
	&e^{-t^2/2}&&\text{ if }a=0, \\ 
	&\frac{\Phi(t+a)-\Phi(t-a)}{\Phi(a)-\Phi(-a)}&&\text{ if }a\in(0,\infty), \\ 
	&1&&\text{ if }a=\infty.   
	\end{alignedat}
	\right. 
\end{equation}
\item $\mathsf{I}_A$ denotes the indicator function of a set $A$, and $\ii{\mathcal A}$ denotes the indicator of an assertion $\mathcal A$; 
\item $\uu$ denotes an arbitrary unit vector in $\R^n$.
\end{itemize}
%Since the sets $A$ to be considered in this paper will be symmetric, 
%
%. Accordingly, concerning definition \eqref{eq:de*}, it will be convenient for us to use here the convention $\sup\emptyset:=0$ (rather than the standard convention $\sup\emptyset:=-\infty$), to avoid repeating the condition $A\ne\emptyset$.

\section{
%Summary% and discussion
Statements of results
}
\label{summary}

\subsection{
Exact upper and lower bounds on the ratio \texorpdfstring{$\E \w(\X-t\uu)/\E \w(\X)$}{}
}
\label{bounds-ratio}

\begin{theorem}\label{th:}
%Let $\X$ be a zero-mean Gaussian random vector in $\R^n$ with a nonsingular covariance matrix $\Si$.  
%Then for any even log-concave function $\w\colon\R^n\to[0,\infty)$ and 
For any 
real $t\ge0$ 
\begin{equation}\label{eq:}
%\begin{aligned}
	e^{-t^2\ip\uu{\Si^{-1}\uu}/2}=r_{t\|\Si^{-1/2}\uu\|}(0)  
	\le\frac{\E \w(\X-t\uu)}{\E \w(\X)}  
	\le r_{t\|\Si^{-1/2}\uu\|}(a_{\Si,\w,\uu})\le1,  
%\end{aligned}	
\end{equation}
where 
%$r_\cdot(\cdot)$ is as %in Lemma~\ref{lem:}
%defined in \eqref{eq:r}, $\|\cdot\|$ is the Euclidean norm on $\R^n$, 
\begin{equation*}
	a_{\Si,\w,\uu}:=\frac{\de^*(\Si^{-1}\uu|A_\w)}{\|\Si^{-1/2}\uu\|} 
	\in[0,\infty], 
\end{equation*}
\begin{equation}\label{eq:A_f}
	A_\w:=%A_{\w,0}=
	\{\xx\in\R^n\colon \w(\xx)>0\}.  
\end{equation}
\end{theorem}

The necessary proofs will be given in Section~\ref{proofs}. 

In the special case of a standard Gaussian random vector, the statement of Theorem~\ref{th:} can be simplified: 

\begin%{theorem*}\label{th:'}
{corollary}\label{cor:w,Z}
%Let $\ZZ$ be a zero-mean Gaussian random vector in $\R^n$ with covariance matrix $I_n$.  
%Then for any even log-concave function $\w\colon\R^n\to[0,\infty)$ and 
For any  
real $t\ge0$
\begin{equation}\label{eq:'}
%\begin{aligned}
	e^{-t^2/2}=r_{t}(0)  
	\le\frac{\E \w(\ZZ-t\uu)}{\E \w(\ZZ)}  
	\le r_{t}(a_{\w,\uu})\le1,  
%\end{aligned}	
\end{equation}
where 
\begin{equation}\label{eq:a_w,v}
	a_{\w,\uu}:=
	a_{I_n,\w,\uu}=\de^*(\uu|A_\w). %\in[0,\infty]. 
\end{equation}
%\end{theorem*}
\end{corollary}

Even though Corollary~\ref{cor:w,Z} is a special case of Theorem~\ref{th:}, it will be seen that, vice versa, Theorem~\ref{th:} can be easily obtained from Corollary~\ref{cor:w,Z}. 

Letting %$\w(\z)=\ii{\z\in A}$ for $\z\in\R^n$
$\w=\mathsf I_A$, %where $\ii\cdot$ denotes the indicator, immediately from
we see that Theorem~\ref{th:} and Corollary~\ref{cor:w,Z} %we obtain 
immediately imply the following two corollaries.  

\begin{corollary}\label{cor:A} 
For any 
real $t\ge0$ 
\begin{equation}\label{eq:A}
\begin{aligned}
	e^{-t^2\ip\uu{\Si^{-1}\uu}/2}=r_{t\|\Si^{-1/2}\uu\|}(0) 
	\le\frac{\P(\X\in t\uu+A)}{\P(\X\in A)}  
	\le r_{t\|\Si^{-1/2}\uu\|}(a_{\Si,A,\uu})\le1,   
\end{aligned}	
\end{equation}
where %$r_\cdot(\cdot)$ is as in Lemma~\ref{lem:}, $\|\cdot\|$ is the Euclidean norm on $\R^n$, 
\begin{equation*}
	a_{\Si,A,\uu}:=\frac{\de^*(\Si^{-1}\uu|A)}{\|\Si^{-1/2}\uu\|} 
	\in[0,\infty].  
\end{equation*}
\end{corollary}

\begin{corollary}\label{cor:A'}
For any 
real $t\ge0$
\begin{equation}\label{eq:A'}
	e^{-t^2/2}=r_{t}(0)  
	\le\frac{\ga_n(t\uu+A)}{\ga_n(A)} 
	\le r_{t}(a_{A,\uu})\le1,   
\end{equation}
where %$\ga_n$ is the standard Gaussian measure over $\R^n$ and 
\begin{equation}\label{eq:a_A}
	a_{A,\uu}:=a_{I_n,A,\uu}=\de^*(\uu|A).  
	%\in[0,\infty].  
\end{equation}
\end{corollary}

%The symmetry of the set $A$ is understood here as the usual condition that $A=-A$. 

%!!! exactness proposition

Let us also present the following statement on the exactness of the lower and upper bounds on $\E \w(\X-t\uu)$ given in Theorem~\ref{th:}: 

\begin{proposition}\label{prop:exact} \rule{0pt}{0pt} %\\ 
\begin{enumerate}[(i)]
	\item For any positive-definite symmetric matrix $\Si$, any unit vector $\uu\in\R^n$, and any real $t\ge0$, the lower bound $e^{-t^2\ip\uu{\Si^{-1}\uu}/2}$ in \eqref{eq:} cannot be replaced by any strictly greater number. 
	\item For any positive-definite symmetric matrix $\Si$, any unit vector $\uu\in\R^n$, any real $t\ge0$, and any $a\in[0,\infty]$, there is an even %log-concave 
	unimodal function $\w\colon\R^n\to[0,\infty)$ such that $a_{\Si,\w,\uu}=a$ and the second equality in \eqref{eq:} turns into the equality. 
\end{enumerate}
\end{proposition}

Letting again %$\w(\z)=\ii{\z\in A}$ for $\z\in\R^n$
$\w=\mathsf I_A$, we see that Proposition~\ref{prop:exact} follows immediately from the corresponding statement on the exactness of the lower and upper bounds on $\P(\X\in t\uu+A)$ given in Corollary~\ref{cor:A}:

\begin{proposition}\label{prop:exactA} \rule{0pt}{0pt} %\\ 
\begin{enumerate}[(i)]
	\item For any positive-definite symmetric matrix $\Si$, any unit vector $\uu\in\R^n$, and any real $t\ge0$, the lower bound $e^{-t^2\ip\uu{\Si^{-1}\uu}/2}$ in \eqref{eq:A} cannot be replaced by any strictly greater number. 
	\item For any positive-definite symmetric matrix $\Si$, any unit vector $\uu\in\R^n$,  any real $t\ge0$, and any $a\in[0,\infty]$, there is a symmetric convex set $A\subseteq\R^n$ such that $a_{\Si,A,\uu}=a$ and the second equality in \eqref{eq:A} turns into the equality. 
\end{enumerate}
\end{proposition}

It follows that the lower and upper bounds given in Corollaries~\ref{cor:w,Z} and \ref{cor:A'} for the special case $\Si=I_n$ are also exact, in the corresponding sense. %s.  

However, the lower bounds in Theorems~\ref{th:} and %\ref{th:}' 
Corollaries~\ref{cor:w,Z}, \ref{cor:A}, and \ref{cor:A'} can be refined as shown in the following subsection.

\subsection{
Bounds on the 
the rate of %the relative decrease of \texorpdfstring{$\E \w(\X-t\uu)$}{} in \texorpdfstring{$t\ge0$}{} 
change of \texorpdfstring{$\E \w(\X-t\uu)$}{} in \texorpdfstring{$t$}{}  
}\label{rel-decr} 

\begin{theorem}\label{th:der}
%Let $\X$ be a zero-mean Gaussian random vector in $\R^n$ with a nonsingular covariance matrix $\Si$.  
%Then for any even log-concave function $\w\colon\R^n\to[0,\infty)$, 
If $\E \w(\X)<\infty$, then for 
%any unit vector $\uu\in\R^n$ and 
any real $t\ge0$ 
\begin{equation}\label{eq:der}
\frac d{dt}\,\E \w(\X-t\uu)\ge-t\,\ip\uu{\Si^{-1}\uu}\,\E \w(\X-t\uu). 	
\end{equation}
\end{theorem}

\begin{corollary}\label{cor:der'}
%Let $\ZZ$ be a zero-mean Gaussian random vector in $\R^n$ with covariance matrix $I_n$.  
%Then for any even log-concave function $\w\colon\R^n\to[0,\infty)$, 
If $\E \w(\ZZ)<\infty$, then for %any unit vector $\uu\in\R^n$ and 
any real $t\ge0$ 
\begin{equation}\label{eq:der'}
\frac d{dt}\,\E \w(\ZZ-t\uu)\ge-t\,\E \w(\ZZ-t\uu). 	
\end{equation}
\end{corollary}

Letting %$\w(\z)=\ii{\z\in A}$ for $\z\in\R^n$
$\w=\mathsf I_A$, %immediately from
we see that Theorem~\ref{th:der} and Corollary~\ref{cor:der'} %we obtain 
immediately imply the following two corollaries.

\begin{corollary}\label{cor:Ader}
%Let $\X$ be a zero-mean Gaussian random vector in $\R^n$ with a nonsingular covariance matrix $\Si$.  
%Then for any symmetric convex set $A\subseteq\R^n$, any unit vector $\uu\in\R^n$, and any real $t\ge0$ 
For %any unit vector $\uu\in\R^n$ and 
any real $t\ge0$ 
\begin{equation}\label{eq:Ader}
\frac d{dt}\,\P(\X\in t\uu+A)\ge-t\,\ip\uu{\Si^{-1}\uu}\,\P(\X\in t\uu+A). 	
\end{equation}
\end{corollary}

\begin{corollary}\label{cor:Ader'}
%For any symmetric convex set $A\subseteq\R^n$, any unit vector $\uu\in\R^n$, and any real $t\ge0$ 
For %any unit vector $\uu\in\R^n$ and 
any real $t\ge0$
\begin{equation}\label{eq:Ader'}
\frac d{dt}\,\ga_n(t\uu+A)\ge-t\,\ga_n(t\uu+A). 	
\end{equation}
\end{corollary}

\begin{remark} The first inequality in \eqref{eq:} can be easily deduced from %Theorem~\ref{th:der}
differential inequality \eqref{eq:der}. Indeed,  
\eqref{eq:der} can be rewritten as %the differential inequality 
$(\ln g)'(t)\ge-ct$ for $t\ge0$, where 
\begin{equation*}
	g(t):=\E \w(\ZZ-t\uu)
\end{equation*}
and $c:=\ip\uu{\Si^{-1}\uu}$. Integrating the differential inequality $(\ln g)'(t)\ge-ct$, we have $g(t)\ge e^{-ct^2/2}g(0)$ for $t\ge0$, 
which is indeed the first inequality in \eqref{eq:}. 
Thus, Theorem~\ref{th:der} and Corollaries~\ref{cor:der'},  \ref{cor:Ader} and \ref{cor:Ader'} are indeed refinements of the corresponding 
lower bounds in Theorem~\ref{th:} and %\ref{th:}' 
Corollaries~\ref{cor:w,Z}, \ref{cor:A}, and \ref{cor:A'}. 
\end{remark} 

For general functions $\w$, which are not necessarily even or %log-concave
unimodal, we have 

\begin{proposition}\label{prop:} 
%Let $\ZZ$ be a zero-mean Gaussian random vector in $\R^n$ with covariance matrix $I_n$, and let 
%Let $\uu$ any unit vector in $\R^n$. 
Let a Borel-measurable function $\w\colon\R^n\to\R$ be such that for some open interval $T\subseteq\R$ and some nonnegative Borel-measurable function $\w_1\colon\R^n\to\R$ we have $\E \w_1(\ZZ)<\infty$ and $|\ip\uu\ZZ\,\w(\ZZ-t\uu)|\le \w_1(\ZZ)$ for all $t\in T$. Then for all $t\in T$ 
\begin{equation*}
	\frac d{dt}\,\E \w(\ZZ-t\uu)=-\ip\uu{\E\ZZ\,\w(\ZZ-t\uu)}. 
\end{equation*}
\end{proposition}

Letting %$\w(\z)=\ii{\z\in A}$ for $\z\in\R^n$
$\w=\mathsf I_A$, immediately from Proposition~\ref{prop:} we obtain the following. % corollary. 

\begin{corollary}\label{cor:c mass}
%Let $\ZZ$ be a zero-mean Gaussian random vector in $\R^n$ with covariance matrix $I_n$, and let 
%Let $\uu$ any unit vector in $\R^n$. 
Let $A$ be any Borel subset of $\R^n$ with $\ga_n(A)\ne0$. Then for all real $t$ one has $\P(\ZZ\in t\uu+A)=\ga_n(t\uu+A)>0$ and 
\begin{equation*}
	\frac d{dt}\,\ln\frac1{\ga_n(t\uu+A)}=\frac d{dt}\,\ln\frac1{\P(\ZZ\in t\uu+A)}=\ip\uu{\E(\ZZ|\,\ZZ\in t\uu+A)};  
\end{equation*}
that is, the rate $\dfrac d{dt}\,\ln\dfrac1{\ga_n(t\uu+A)}=-\dfrac{\frac d{dt}\ga_n(t\uu+A)}{\ga_n(t\uu+A)}$ of the relative decrease of\ \,$\ga_n(t\uu+A)$ in $t$ equals the $\uu$-coordinate %\break 
$\ip\uu{\E(\ZZ|\,\ZZ\in t\uu+A)}$ of the center $\E(\ZZ|\,\ZZ\in t\uu+A)$ of the standard Gaussian mass over the set $t\uu+A$. 
\end{corollary}

Now Corollary~\ref{cor:Ader'} can be restated as follows: 

\begin{corollary}\label{cor:c mass slow}
Suppose that the conditions of Corollary~\ref{cor:c mass} hold and, in addition, the set $A$ is symmetric and convex. Then for all real $t\ge0$ 
\begin{equation*}
	\ip\uu{\E(\ZZ|\,\ZZ\in t\uu+A)}\le t; 
\end{equation*}
that is, when the symmetric convex set $A$ is shifted by the vector $t\uu$, 
the $\uu$-coordinate of the center of the standard Gaussian mass over the set $t\uu+A$ increases by no more than $t$. 
\end{corollary}

%!!!??? new even for balls, noncentral $\chi^2$

%\medskip
%\hrule
%\medskip

\subsection{
Hypothesis testing 
}\label{H} 

Corollary~\ref{cor:A} can be restated in terms of hypothesis testing:

\begin{corollary}\label{cor:H}
Let $\Y$ be a Gaussian random vector in $\R^n$ with an unknown mean $\muu$ and a known nonsingular covariance matrix $\Si$. 
We test the null hypothesis $H_0\colon\muu=\0$ versus the alternative $H_1\colon\muu=\th\uu%\ne\0
$ for real $\th>0$, using the test $\de(\Y):=\ii{\Y\notin A}$ with a symmetric convex set $A\subseteq\R^n$, so that the null hypothesis is rejected if and only if $\Y\notin A$, and the size of the test is $\al:=\P(\X\notin A)$ (where $\X$ is as in Corollary~\ref{cor:A}). Then for the power 
\begin{equation*}
	\be_\de(\th)=\P_\th(\Y\notin A)=\P(\X\notin A-\th\uu)=1-\P(\X\in A-\th\uu)
\end{equation*}
of the test $\de$ at any alternative $\muu=\th\uu%\ne\0
$ for a real $\th>0$ we have 
\begin{equation}\label{eq:H}
%\begin{aligned}
	1-e^{-\th^2\ip\uu{\Si^{-1}\uu}/2}(1-\al)  
	\ge\be_\de(\th) \\ 
	\ge 1-r_{\th\|\Si^{-1/2}\uu\|}(a_{\Si,A,\uu})(1-\al)\ge\al;     
%\end{aligned}	
\end{equation}
here $\P_\th$ denotes the probability computed assuming that $\muu=\th\uu$ is the true mean of $\Y$. 
\end{corollary}

%\cite{prekopa71}
%this property because they satisfy the conditions of the main theorem.
%Inequality (1. 5) has an important consequence, namely that the P measure
%of the parallel shifts of a convex set is a logarithmic concave function of the shift
%vector. 

%\cite{leindler72,leindler72a}

%survey \cite{gardner02} 
 
\section{Proofs} 
\label{proofs}

Of the first four results stated in Section~\ref{summary} -- Theorem~\ref{th:} and %\ref{th:}' 
Corollaries~\ref{cor:w,Z}, \ref{cor:A}, and \ref{cor:A'} -- Corollary~\ref{cor:A'} is formally the least general. However, we shall prove Corollary~\ref{cor:A'} first. From there, it will not be hard to deduce the more general Corollary~\ref{cor:w,Z} and then in turn Theorem~\ref{th:}, which latter immediately yields Corollary~\ref{cor:A} as a special case. 
Then a proof of Proposition~\ref{prop:exactA} will be given.

After that, % , similarly, 
we will prove Corollary~\ref{cor:der'} and Theorem~\ref{th:der}, 
%Corollary~\ref{cor:Ader'}, Corollary~\ref{cor:der'}, and Theorem~\ref{th:der}, 
in this order. Corollaries~\ref{cor:Ader} and \ref{cor:Ader'} will then follow immediately. 

A proof of Proposition~\ref{prop:} will conclude this section. 

\bigskip 
 
%, and then it will not be hard to obtain the more general results: Theorems~\ref{th:} and \ref{th:}' and Corollary~\ref{cor:A}. 
%
%In this section we shall first prove Corollary~\ref{cor:w,Z}. Theorem~\ref{th:} is then easily deduced from Corollary~\ref{cor:w,Z}. 
%
%After that, we shall prove Corollary~\ref{cor:der'}. Theorem~\ref{th:der} is then easily deduced from Corollary~\ref{cor:der'}. 

To prove %Corollary~\ref{cor:w,Z}
Corollary~\ref{cor:A'}, we shall need 

\begin{lemma}\label{lem:}
The expression $r_t(a)$, defined in \eqref{eq:r}, is continuous and nondecreasing in $a\in[0,\infty]$, for each $t\in[0,\infty)$.   
\end{lemma}

\begin{proof}%[Proof of Lemma~\ref{lem:}]
The case $t=0$ is trivial. Fix now any $t\in(0,\infty)$.  
That $r_t(a)$ is continuous in $a$ at $a=\infty$ is obvious. It is also obvious that $r_t(a)$ is continuous in $a$ at each point $a\in(0,\infty)$. That $r_t(a)$ is continuous in $a$ at $a=0$ follows by the l'Hospital rule. 

It remains to show that $r_t(a)$ is increasing in $a\in(0,\infty)$. For such $a$, we have 
\begin{equation*}
	r_t(a)=\frac{\psi_t(a)}{\psi_0(a)}, 
\end{equation*}
where $\psi_t(a):=\Phi(t+a)-\Phi(t-a)$. 
%As usual, let $\vpi$ denote the standard normal probability density function: 
%\begin{equation*}
%	\vpi:=\Phi'. 
%\end{equation*}
Note that $\psi_t(0+)=\psi_0(0+)=0$ and the ``derivative ratio''
\begin{equation*}
	%\rho_t(a):=
	\frac{\psi'_t(a)}{\psi'_0(a)}=\frac{\vpi(t+a)+\vpi(t-a)}{2\vpi(a)}
	=e^{-t^2/2}\cosh ta
\end{equation*}
is increasing in $a\in(0,\infty)$.  
So, by Theorem~B (stated in Section~\ref{intro}), 
%the so-called special-case l'Hospital-type rule for monotonicity \cite[Proposition 4.1]{pin06}, 
$r_t(a)$ is increasing in $a\in(0,\infty)$. The proof of Lemma~\ref{lem:} is complete. 
\end{proof}

\begin{proof}[Proof of %Corollary~\ref{cor:w,Z}
Corollary~\ref{cor:A'}]
In view of \eqref{eq:a_A} and because the 
set $A$ is symmetric,  
\begin{equation}\label{eq:a=}
	a_{A,\uu}=\de^*(\uu|A)=\sup\{\ip\z\uu\colon\z\in A\} 
	=\sup\{|\ip\z\uu|\colon\z\in A\}.   
\end{equation}
By the spherical symmetry of the standard Gaussian measure $\ga_n$, without loss of generality $\uu$ equals $\e_1$, the first vector of the standard basis of $\R^n$. So, denoting by $\vpi_k$ the density of the standard Gaussian measure $\ga_k$ over $\R^k$ (with respect to the Lebesgue measure over $\R^k$), %with $\vpi:=\vpi_1$, 
we have 
\begin{align}
	g_{A,\uu}(t):=%\E \w(\ZZ-t\uu)
	\ga_n(t\uu+A)&=\int_{%\R^n
	t\uu+A}d\z\,\vpi_n(\z)%\w(\z-t\uu) 
	\label{eq:=g=} 
	\\ 
&	=\int_%{\R^n}
A d\x\,\vpi_n(\x+t\uu)%\w(\x) 
\label{eq:phi(x+tu)} \\ 
&	=\int_{\R}dx\,\vpi(x+t)h_{A,\uu}(x) \notag \\ %\label{eq:g=}
&	=\int_{-a_{A,\uu}}^{a_{A,\uu}}dx\,\vpi(x+t)h_{A,\uu}(x), \label{eq:g=}
\end{align}
where 
\begin{equation*}
	h_{A,\uu}(x):=\int_{\R^{n-1}}d\y\,\vpi_{n-1}(\y)%\w(x\uu+\y) 
	\,\mathsf I_A(x\uu+\y) 
\end{equation*}
and the orthogonal complement $\{%\ww
\y\in\R^n\colon\ip%\ww
\y\uu=0\}$ of the vector $\uu=\e_1$ is identified with $\R^{n-1}$;  
equality \eqref{eq:g=} holds because, if $|x|>a_{A,\uu}$, then for all $\y\in\R^{n-1}$ we have $|\ip{x\uu+\y}\uu|=|x|>a_{A,\uu}$; so, by \eqref{eq:a=}, %$\w(x\uu+\y)=0$
$x\uu+\y\notin A$ for all $\y\in\R^{n-1}$, whence $h_{A,\uu}(x)=0$. 
 
The functions 
$\R\times\R^{n-1}\ni(x,\y)\mapsto %\w(x\uu+\y)
\mathsf I_A(x\uu+\y)\in[0,\infty)$ and $\vpi_{n-1}$ are even and log concave, and hence so is the function 
$$\R\times\R^{n-1}\ni(x,\y)\mapsto\vpi_{n-1}(\y)%\w(x\uu+\y)
\mathsf I_A(x\uu+\y)\in[0,\infty).$$ 
Therefore, 
in view of Theorem~A, 
%the Pr\'ekopa--Leindler theorem (see e.g.\ \cite{brasc-lieb76}), 
the function $h_{A,\uu}\colon\R\to[0,\infty)$ is also even and log concave, and hence unimodal; it also follows that $h_{A,\uu}$ is continuous %everywhere except at most at two points
on the interval $(-a_{A,\uu},a_{A,\uu})$. 
So, there is a (unique, nonnegative, finite) Borel measure $\mu_{A,\uu}$ over the interval $(0,a_{A,\uu}]$ such that $\mu_{A,\uu}\big((x,a_{A,\uu}]\big)=h_{A,\uu}(x)$ for all $x\in[0,a_{A,\uu})$, and then for $x\in(-a_{A,\uu},a_{A,\uu})$ we have 
\begin{align*}
	h_{A,\uu}(x)&=h_{A,\uu}(|x|)=\mu_{A,\uu}\big((|x|,a_{A,\uu}]\big)=\int_{(0,a_{A,\uu}]} \mu_{A,\uu}(da)\,\mathsf I\{a>|x|\} \\ 
	&=\int_{(0,a_{A,\uu}]} \mu_{A,\uu}(da)\,h_a(x),  
\end{align*}
where 
\begin{equation*}
	h_a(x):=\ii{|x|<a}. 
\end{equation*}  
So, by \eqref{eq:g=} and the Fubini theorem, %for all real $t$ 
\begin{equation}\label{eq:g(t)=}
	g_{A,\uu}(t)=\int_{(0,a_{A,\uu}]} \mu_{A,\uu}(da)\,\int_{-a_{A,\uu}}^{a_{A,\uu}} dx\,\vpi(x+t)h_a(x)
	=\int_{(0,a_{A,\uu}]} \mu_{A,\uu}(da)\,g_a(t), 
\end{equation}
where, for $a\in(0,a_{A,\uu}]$,  
\begin{equation}\label{eq:g_a}
	g_a(t):=
	\int_{-a}^a dx\,\vpi(x+t)%=\Phi(a+t)-\Phi(-a+t)
	=r_t(a)g_a(0), 
\end{equation}
in view of \eqref{eq:r}. 
So, by Lemma~\ref{lem:}, 
\begin{equation*}
	r_t(0)g_a(0)\le g_a(t)\le r_t(a_{A,\uu})g_a(0)
\end{equation*}
for $a\in(0,a_{A,\uu}]$, whence, by \eqref{eq:g(t)=}, 
\begin{equation}\label{eq:rg<g<rg}
	r_t(0)g_{A,\uu}(0)\le g_{A,\uu}(t)\le r_t(a_{A,\uu})g_{A,\uu}(0). 
\end{equation}
In view of \eqref{eq:=g=}, % and \eqref{eq:t,u}, 
inequalities \eqref{eq:rg<g<rg} are the same as the first two inequalities in \eqref{eq:%
A'}. The equality and the third inequality in \eqref{eq:%
A'} follow by \eqref{eq:r} and Lemma~\ref{lem:}. %Corollary~\ref{cor:w,Z} 
Corollary~\ref{cor:A'} is now proved. 
\end{proof}

The following remark will be used in the proof of Corollary~\ref{cor:w,Z}. 

\begin{remark}\label{rem:mix}
We have
\begin{equation}\label{eq:mix}
	\w(\x)=\int_0^\infty \w_c(\x)\,dc
\end{equation}
for all $\x\in\R^n$, where 
\begin{equation}\label{eq:A_w,c}
	\w_c:=\mathsf{I}_{A_{\w,c}},\quad A_{\w,c}:=\{\x\in\R^n\colon\w(\x)%\ge 
	>c\}.  
\end{equation}
%and $\mathsf{I}$ denotes the indicator. 
Note also that the functions $\w_c$ are log concave. 
Thus, the unimodal function $\w$ is a mixture of log concave functions $\w_c$; moreover, the functions $\w_c$ are even whenever the function $\w$ is even. 
% (that is, whenever $\w(-\x)=\w(\x)$ for all $\x\in\R^n$). 
\end{remark}

\begin{proof}[Proof of Corollary~\ref{cor:w,Z}]
For any $c\ge0$, by \eqref{eq:A_w,c} and \eqref{eq:A_f}, $A_{\w,c}\subseteq A_{\w,0}=A_\w$, whence, by \eqref{eq:a_A} and \eqref{eq:a_w,v}, 
$a_{A_{\w,c},\uu}\le a_{A_\w,\uu}=a_{\w,\uu}$. Now Lemma~\ref{lem:} yields 
$r_{t}(a_{A_{\w,c},\uu})\le r_{t}(a_{\w,\uu})$. 
Hence, by %\eqref{eq:mix}, \eqref{eq:A_w,c},
Remark~\ref{rem:mix} and the second inequality in \eqref{eq:A'}, 
\begin{align}
	\E \w(\ZZ-t\uu)&=\int_0^\infty \E \w_c(\ZZ-t\uu)\,dc \notag \\ 
	&=\int_0^\infty \P(\ZZ\in t\uu+A_{\w,c})\,dc \notag \\ 
	&=\int_0^\infty \ga_n(t\uu+A_{\w,c})\,dc \label{eq:=int dc} \\ 
	&\le\int_0^\infty r_{t}(a_{A_{\w,c},\uu}) \ga_n(A_{\w,c})\,dc \notag \\ 
	&\le r_{t}(a_{\w,\uu})\int_0^\infty\ga_n(A_{\w,c})\,dc
	=r_{t}(a_{\w,\uu})\E \w(\ZZ), \notag 
\end{align}
which proves the second inequality in \eqref{eq:'}. The proofs of the other inequalities in \eqref{eq:'} and of the equality there are similar and even somewhat simpler. 
\end{proof}

\begin{proof}[Proof of Theorem~\ref{th:}]
Given $\X$, $\Si$, $\w$, $\uu$, and $t$ as in the statement of Theorem~\ref{th:}, define $\tX$, $\tf$, $\tuu$, and $\ttt$ as follows: $\tX:=\Si^{-1/2}\X$, $\tf(\tx):=\w(\Si^{1/2}\tx)$ for $\tx\in\R^n$, $\tuu:=\Si^{-1/2}\uu/\|\Si^{-1/2}\uu\|$, and $\ttt:=t\|\Si^{-1/2}\uu\|$. 
Applying now Corollary~\ref{cor:w,Z} with, respectively, $\tX$, $\tf$, $\tuu$, and $\ttt$ in place of $\ZZ$, $\w$, $\uu$, and $t$ there, we obtain Theorem~\ref{th:}. 
\end{proof}

%!!!!!!!!! stopped here 4/13/19, 10:44

\begin{proof}[Proof of Proposition~\ref{prop:exactA}]
As in the %proofs of Theorems~\ref{th:} and \ref{th:der}
proof of Theorems~\ref{th:}, the consideration can be easily reduced to the case $\Si=I_n$, so that $\X=\ZZ$, a standard Gaussian random vector. Take then indeed any 
$a\in[0,\infty]$ and let 
\begin{equation*}
	A:=\{z\in\R^n\colon|\ip\z\uu|\le a\}. 
\end{equation*} 
Then 
\begin{equation*}
	a_{\Si,A,\uu}=a_{I_n,A,\uu}=\de^*(\uu|A)=a
\end{equation*}
by \eqref{eq:de*}. 
Also, for $a\in[0,\infty)$, 
\begin{equation*}
\begin{aligned}
	\P(\ZZ\in t\uu+A)&=\P(|\ip{\ZZ-t\uu}\uu|\le a) \\ 
	&=\P(t-a\le \ip\ZZ\uu\le t+a) \\ 
	&=\Phi(t+a)-\Phi(t-a) 
	=r_{t}(a)\P(\ZZ\in A)
\end{aligned}	
\end{equation*}
by \eqref{eq:r}, so that the second equality in \eqref{eq:A} (with $I_n$ and $\ZZ$ in place of $\Si$ and $\X$) turns into the equality; the case $a=\infty$ is even simpler than this. This proves part~(ii) of Proposition~\ref{prop:exactA}. 

Part (i) of it now follows because 
\begin{equation*}
	\frac{\P(\ZZ\in t\uu+A)}{\P(\ZZ\in A)}
	=r_{t}(a)\underset{a\downarrow0}\longrightarrow r_{t}(0)=e^{-t^2/2},
\end{equation*}
by Lemma~\ref{lem:}.  
\end{proof}

To prove Corollary~\ref{cor:der'}, we shall need 

\begin{lemma}\label{lem:der}
Recall the definition of $g_a(t)$ in \eqref{eq:g_a}. We have 
\begin{equation}\label{eq:lem:der}
	0\ge g'_a(t)\ge-tg_a(t)
\end{equation}
for all $a\in[0,\infty]$ and $t\in[0,\infty)$.   
\end{lemma}

\begin{proof}%[Proof of Lemma~\ref{lem:}]
Fix any $a\in[0,\infty]$. If $a=\infty$, then $g_a(t)=1$ for all $t$ and hence \eqref{eq:lem:der} is obvious. So, without loss of generality $a\in[0,\infty)$, and then 
\begin{equation}\label{eq:g_a,g'_a}
	g_a(t)=\Phi(a+t)-\Phi(-a+t)\quad\text{and}\quad 
	g'_a(t)=\vpi(a+t)-\vpi(-a+t)
\end{equation}
for all $t$, which yields the first inequality in \eqref{eq:lem:der}. 

Also, for $t=0$ \eqref{eq:lem:der} is trivial. Hence, without loss of generality $t\in(0,\infty)$, and then the second inequality in 
\eqref{eq:lem:der} can be rewritten as $\la_a(t)\ge0$, where 
\begin{equation*}
	\la_a(t):=g_a(t)+g'_a(t)/t=\Phi(a+t)-\Phi(-a+t)+\frac{\vpi(a+t)-\vpi(-a+t)}t. \label{eq:psi}
\end{equation*}
But
\begin{equation*}
	\la'_a(t)=(\tanh at - a t) \,\frac{2e^{at}}{t^2}\, \vpi(a + t)  \cosh at\le0,
\end{equation*}
since $\tanh 0=0$ and $\tanh'=1/\cosh^2\le1$. So, $\la_a(t)$ is decreasing in $t>0$, to $\la_a(\infty-)=0$. So, for $t>0$ we do have $\la_a(t)\ge0$, which completes the proof of the second inequality in \eqref{eq:lem:der}. 
%The first inequality in \eqref{eq:lem:der} easily follows from \eqref{eq:g_a,g'_a}. 
%
Lemma~\ref{lem:der} is now proved.   
\end{proof}

%\newpage

\begin{proof}[Proof of Corollary~\ref{cor:der'}]
By \eqref{eq:=int dc}, \eqref{eq:=g=}, and \eqref{eq:g(t)=}, 
\begin{equation}\label{eq:=intint}
\begin{aligned}
	\E \w(\ZZ-t\uu)=\int_0^\infty dc\;\ga_n(t\uu+A_{\w,c})  
 &=\int_0^\infty dc\; g_{A_{\w,c},\uu}(t) \\ 
&=\int_0^\infty dc\; \int_{(0,a_{A_{\w,c},\uu}]} \mu_{A_{\w,c},\uu}(da)\,g_a(t).  
\end{aligned}	
\end{equation}
%with the last equality following by \eqref{eq:=intint}. 
By Lemma~\ref{lem:der}, for all $a\in[0,\infty]$ and $t\in[0,\infty)$ we have 
\begin{equation}\label{eq:lem:derr}
	0\ge g'_a(t)\ge-tg_a(t)\ge-tg_a(0),  
\end{equation}
whence 
\begin{align*}
	\int_0^\infty dc\; \int_{(0,a_{A_{\w,c},\uu}]} \mu_{A_{\w,c},\uu}(da)\,|g'_a(t)|
	&\le t\int_0^\infty dc\; \int_{(0,a_{A_{\w,c},\uu}]} \mu_{A_{\w,c},\uu}(da)\,g_a(0) \\ 
	&=t\E \w(\ZZ)<\infty. 
\end{align*}
So, using \eqref{eq:=intint} and \eqref{eq:lem:derr} again, together with the standard rule of the differentiation of an integral with respect to a parameter -- see e.g.\ \cite[Theorem~(2.27)(b)]{folland}, we see that for all real $t\ge0$
\begin{equation*}
\begin{aligned}
	\frac d{dt}\,\E \w(\ZZ-t\uu)&=\int_0^\infty dc\; \int_{(0,a_{A_{\w,c},\uu}]} \mu_{A_{\w,c},\uu}(da)\,g'_a(t) \\ 
	&\ge-t\int_0^\infty dc\; \int_{(0,a_{A_{\w,c},\uu}]} \mu_{A_{\w,c},\uu}(da)\,g_a(t)
	=-t\E \w(\ZZ-t\uu),   
\end{aligned}	
\end{equation*}
which completes the proof of Corollary~\ref{cor:der'}. 
%
%This follows immediately from \eqref{eq:g(t)=}, the definition of $g(t)$ in \eqref{eq:=g=}, and Lemma~\ref{lem:der}. 
\end{proof}

\begin{proof}[Proof of Theorem~\ref{th:der}]
Given $\X$, $\Si$, $\w$, $\uu$, and $t$ as in the statement of Theorem~\ref{th:der}, define $\tX$, $\tf$, $\tuu$, and $\ttt$ as follows: $\tX:=\Si^{-1/2}\X$, $\tf(\tx):=\w(\Si^{1/2}\tx)$ for $\tx\in\R^n$, $\tuu:=\Si^{-1/2}\uu/\|\Si^{-1/2}\uu\|$, and $\ttt:=t\|\Si^{-1/2}\uu\|$. 
Applying now Corollary~\ref{cor:der'} with, respectively, $\tX$, $\tf$, $\tuu$, and $\ttt$ in place of $\ZZ$, $\w$, $\uu$, and $t$ there, we obtain Theorem~\ref{th:der}. 
\end{proof}

\begin{proof}[Proof of Proposition~\ref{prop:}]
Using %\eqref{eq:phi(x+tu)}, 
the domination condition involving the function $\w_1$ and (again) the rule of the differentiation of an integral with respect to a parameter, for all $t\in T$ we have
\begin{align*}
	\frac d{dt}\,\E \w(\ZZ-t\uu)	
	&=\frac d{dt}\,\int_{\R^n}d\z\;\vpi_n(\z)\w(\z-t\uu) \\ 
	&=\frac d{dt}\,\int_{\R^n}d\x\;\vpi_n(\x+t\uu)\w(\x) \\ 
	&=\int_{\R^n}d\x\;\frac d{dt}\,\vpi_n(\x+t\uu)\w(\x) \\ 
	&=-\int_{\R^n}d\x\,\vpi_n(\x+t\uu)\,\ip{\x+t\uu}\uu\,\w(\x) \\ 
	&=-\int_{\R^n}d\z\,\vpi_n(\z)\,\ip\uu\z\,\w(\z-t\uu) \\
	&=-\ip\uu{\E\ZZ\,\w(\ZZ-t\uu)}, 
\end{align*}
as claimed. 
%To obtain the first equality in the latter multi-line display, we also used the domination condition involving the function $\w_1$ and the standard rule of the differentiation of an integral with respect to a parameter -- see e.g.\ \cite[Theorem~(2.27)(b)]{folland}. 
\end{proof}

%%%%\bibliographystyle{splncsnat}
%%%%\bibliographystyle{rss}
%%%%\bibliographystyle{abbrv}
%%%%\bibliographystyle{imsart-number}
%%%\bibliographystyle{amsplain}
%%%%\bibliography{citations.nodoi}
%%%
%%%\bibliography{P:/pCloudSync/mtu_pCloud_02-02-17/bib_files/citations10.13.18a}

\def\cprime{$'$} \def\polhk#1{\setbox0=\hbox{#1}{\ooalign{\hidewidth
  \lower1.5ex\hbox{`}\hidewidth\crcr\unhbox0}}}
  \def\polhk#1{\setbox0=\hbox{#1}{\ooalign{\hidewidth
  \lower1.5ex\hbox{`}\hidewidth\crcr\unhbox0}}}
  \def\polhk#1{\setbox0=\hbox{#1}{\ooalign{\hidewidth
  \lower1.5ex\hbox{`}\hidewidth\crcr\unhbox0}}} \def\cprime{$'$}
  \def\polhk#1{\setbox0=\hbox{#1}{\ooalign{\hidewidth
  \lower1.5ex\hbox{`}\hidewidth\crcr\unhbox0}}} \def\cprime{$'$}
  \def\polhk#1{\setbox0=\hbox{#1}{\ooalign{\hidewidth
  \lower1.5ex\hbox{`}\hidewidth\crcr\unhbox0}}} \def\cprime{$'$}
  \def\cprime{$'$}
\providecommand{\bysame}{\leavevmode\hbox to3em{\hrulefill}\thinspace}
\providecommand{\MR}{\relax\ifhmode\unskip\space\fi MR }
\providecommand{\MRhref}[2]{%
  \href{http://www.ams.org/mathscinet-getitem?mr=#1}{#2}
}
\providecommand{\href}[2]{#2}

\end{document}